\documentclass[reqno]{amsart}
\usepackage{amssymb, amsmath}
\usepackage{natbib}
\usepackage{lscape}
\usepackage{dsfont}
\usepackage{mathrsfs}
\usepackage{bbm}

\usepackage[pdftex,plainpages=false,colorlinks,hyperindex,bookmarksopen,linkcolor=red,citecolor=blue,urlcolor=blue]{hyperref}

\DeclareMathAlphabet{\mathpzc}{OT1}{pzc}{m}{it}

\bibpunct{[}{]}{;}{n}{,}{,}

\newtheorem{te}{Theorem}[section]
\newtheorem{defin}[te]{Definition}
\newtheorem{os}[te]{Remark}
\newtheorem{prop}[te]{Proposition}
\newtheorem{lem}[te]{Lemma}

\numberwithin{equation}{section}

\allowdisplaybreaks

\newcommand {\puntomio} {\mathbin{\vcenter{\hbox{\scalebox{.45}{$\bullet$}}}}}

\def \l { \left( }
\def \r {\right) }
\def \ll { \left\lbrace }
\def \rr { \right\rbrace }

\begin{document}

	\title[]{L\'evy mixing related to distributed order calculus, subordinators and slow diffusions}
	\author{Bruno Toaldo} 
	\address{Department of Statistical Sciences, Sapienza - University of Rome}
	\email {bruno.toaldo@uniroma1.it}
	\date{\today}

		\begin{abstract}
The study of distributed order calculus usually concerns about fractional derivatives of the form $\int_0^1 \partial^\alpha u \, m(d\alpha)$ for some measure $m$, eventually a probability measure. In this paper an approach based on L\'evy mixing is proposed. Non-decreasing L\'evy processes associated to L\'evy triplets of the form $\l a(y), b(y) , \nu(ds, y) \r$ are considered and the parameter $y$ is randomized by means of a probability measure. The related subordinators are studied from different point of views. Some distributional properties are obtained and the interplay with inverse local times of Markov processes is explored. Distributed order integro-differential operators are introduced and adopted in order to write explicitly the governing equations of such processes. An application to  slow diffusions (delayed Brownian motion) is discussed.
\end{abstract}
	
	\maketitle

\tableofcontents

\section{Introduction}
Distributed order calculus usually concerns about the distributed order fractional derivative i.e.
\begin{align}
\frac{^pd^\beta}{d x^\beta} u(x) \, = \,  \int_0^1 \frac{d^\beta}{dx^\beta} u(x) p(d\beta),
\label{dofd}
\end{align}
where $\frac{d^\beta}{dx^\beta}$ can be a fractional derivative in the Riemann-Liouville sense as well as a Dzerbayshan-Caputo derivative and $p(d\beta)$ is a measure (eventually a probability measure). For the definitions of fractional derivatives see, for example, \citet{kill}.
 
In this paper we develop an approach to distributed-order calculus based on L\'evy mixing (see \citet{bn} for L\'evy mixing). We consider functions of the form
\begin{align}
f(\lambda, y) \, = \, a(y) + b(y) \lambda + \int_0^\infty \l 1-e^{-\lambda s} \r \nu(ds, y)
\label{prese}
\end{align}
such that $\lambda \to f(\lambda, y)$ is a Bernstein function for all $y$ in some Polish space $E$. Under suitable integrability conditions on the family $\ll \nu(ds, y) \rr_{y \in E}$ of L\'evy measures parametrized by $y$ and on $a(y)$ and $b(y)$ we then consider the function
\begin{equation}
\int_W f(\lambda, y) p(dy) \, = \, \int_W \l a(y) + b(y)\lambda \r p(dy) + \int_0^\infty \l 1-e^{-\lambda s} \r \int_W \nu(ds, y) p(dy)
\label{mixfirsttime}
\end{equation}
for a probability measure $p$ on $W \subseteq E$.
Since $p$ is a probability measure we will write for the sake of simplicity
\begin{align}
\int_W f(\lambda, y) p(dy) \, = \,\mathds{E} f(\lambda, Y)
\end{align}
for a r.v. $Y$ with law $p$ on $W\subseteq E$.
Such a procedure is a particular case of the so-called L\'evy mixing which has been recently studied in a systematic way by \citet{bn} in case the measure $p$ is not necessarily a probability measure. Among other things the authors pointed out that L\'evy mixing arises naturally from stochastic integral representation of processes.

One of the goals of the present paper is to study subordinators with Laplace exponent as in \eqref{mixfirsttime} and we call such processes \emph{distributed order} subordinators. We study the transition probabilities of such subordinators by relating them with the most important classes of convolution semigroups. We explore the connection with the inverse local times of Markov processes and we study the right-continuous inverses i.e. the hitting-times. Distributed order calculus comes into play for writing the governing equations of distributed order subordinators and of their hitting-times. From this point of view we have inserted the theory of distributed order calculus (based on operators of the form \eqref{dofd}) in a more general unifying framework.

We introduce the distributed order integro-differential operator
\begin{align}
_c\mathcal{D}^{f, p}_t u(t) \, = \, \mathds{E}  b(Y) \frac{d}{dt}u(t) + \frac{d}{dt} \int_c^t u(s) \mathds{E} \bar{\nu}(t-s, Y) ds, \qquad c >0,
\label{qui}
\end{align}
where $\bar{\nu}(s, y) =   a(y)+ \nu((s, \infty), y)$ and $Y$ is a r.v. with law $p$ on $W\subseteq E$. Denote by $\mathcal{L} \left[ u (\puntomio) \right] (\lambda) = \widetilde{u}(\lambda)$ the Laplace transform of the function $u$, we note that the operator \eqref{qui} has Laplace symbol
\begin{align}
\mathcal{L} \left[ \, _0\mathcal{D}_t^{f, p} u(t) \right] (\lambda) \, = \, \mathds{E} f(\lambda, Y) \widetilde{u}(\lambda) - \mathds{E} b(Y) u(0)
\end{align}
where $\lambda \to \mathds{E} f(\lambda, Y)$ is a Bernstein function as in \eqref{mixfirsttime}, and permits us to write explicitly the governing equations of the above distributed order subordinators with Laplace exponent \eqref{mixfirsttime}.
Then we introduce the distributed order operator
\begin{align}
_c\mathsf{D}^{f, p}_tu(t) \, = \, \mathds{E}  b(Y) \frac{d}{dt}u(t) + \frac{d}{dt} \int_0^t u(s) \mathds{E} \bar{\nu}(t-s, Y) ds - \mathds{E}  \bar{\nu}(t-c, Y) u(c)
\end{align}
and we discuss an application to the diffusion equation
\begin{align}
_0\mathsf{D}^{f, p}_t q(x, t)\, = \, \Delta q(x, t), \qquad x \in \mathbb{R}^d, t>0.
\label{diffintro}
\end{align}
Under suitable assumptions we show that the mean square displacement
\begin{align}
\mathscr{M}(t) \, = \, \int_{\mathbb{R}^d} \, |x-u|^2 \, q(x-u, t) du,
\label{msdintro}
\end{align}
where $q$ is the fundamental solution to \eqref{diffintro}, behaves as 
\begin{align}
1\bigg/\int_W f(1/t, y)p(dy) \textrm{ for } t \to \infty.
\label{simintro}
\end{align}
Usually a diffusion is said to be slow (or subdiffusion) if $\mathscr{M}(t) \sim Ct^\alpha$ with $\alpha < 1$, $C>0$. Here we study
\begin{align}
\lim_{t \to \infty} \frac{t}{\mathscr{M}(t)}
\label{rapptintro}
\end{align}
by using \eqref{simintro} and we show that \eqref{rapptintro} can not be zero but may be either finite or infinite. Actually in the most common cases \eqref{rapptintro} is infinite proving that diffusions related to \eqref{diffintro} are chiefly subdiffusive. We examine a particular case (related to distributed order fractional calculus) in which the mean square displacement is $\sim C \log t$. This last situation was discussed in \citet{chechprimo, chechsecondo, chechterzo, kochudo} and is related to the so-called ultraslow diffusions.
\subsection{Some background information on distributed and fractional order calculus}
Operators of the form \eqref{dofd} have been introduced for the first time in \citet{caputo}. Successively a rigorous mathematical theory has been developed in \citet{kochudo} and applied to diffusion equations of the form
\begin{align}
\frac{^p\partial^\beta}{\partial t^\beta} u(x, t) \, = \, \Delta u(x, t) + f(x, t),
\end{align}
which describe some particular anomalous diffusions (for anomalous diffusions the reader can consult \citet{metkla1, metkla2}). In particular \citet{kochudo} deals with the distributed order fractional derivative
\begin{align}
\frac{d}{dt} \int_0^t u(s) \int_0^1 \frac{(t-s)^{-\beta}}{\Gamma(1-\beta)} p(d\beta) ds - \int_0^1 \frac{t^{-\beta}}{\Gamma(1-\beta)} u(0) p(d\beta)
\label{introdokochu}
\end{align}
which is defined for a continuous function $u$. For different forms of distributed order fractional operators the reader can consult \citet{lorenzo} and the references therein. The study of boundary value problems for time-fractional distributed order diffusion equations can be found in \citet{luccio}.
In physics the literature is vast. For example in \citet{chechprimo, chechsecondo} different types of diffusions related to distributed order operators have been studied and a probabilistic description of slow diffusions can be found in \citet{meerstmod}. In \citet{mainpagn} the relationship between distributed order diffusions and the Fox function have been explored. An application to the fractional relaxation can be found in \citet{mainal} and the asymptotic behaviour of the solution to a distributed order relaxation equation has been studied in \citet{kochuphys}. 

Much attention on the probabilistic point of view has been adopted recently in \citet{luisado, meerdo}. In particular in \cite{luisado} an approach based on time-changed processes is investigated. The authors take inspiration from the known results concerning the interplay between fractional equations and time-changed processes and extend the results to the distributed order case. Such an interplay has been explored firstly in work such as \citet{allouba, baem, orsptrf, orsann, saiche, zasla}. A classical result (\cite{baem}) states, roughly speaking, that a L\'evy motion time-changed with the hitting-time process of a $\beta$-stable subordinator is the stochastic solution to a time-fractional Cauchy problem. A $\beta$-stable subordinator is a non-decreasing L\'evy process with Laplace exponent $\lambda^\beta$, $\beta \in (0,1)$.
Such results have been deeply investigated (see, for example, \citet{meerann} for fractional Cauchy problems in bounded domains) and has been successively extended as follows. In \citet{meertri} the authors considered continuous time random walks (CTRW) time-changed with a renewal process. They pointed out that the limit of such a CTRW is a L\'evy process $\mathpzc{L}$ time-changed with the hitting-time of a subordinator $\sigma^f(t)$, $t>0$, and that its density is the fundamental solution to the abstract Cauchy problem
\begin{align}
f \l \partial_t \r q(x, t) \, = \, A q(x, t) + \delta(x) \nu(t, \infty),
\label{acp}
\end{align}
in the mild sense. In \eqref{acp} $A$ is the generator of $\mathpzc{L}$, $f$ is the Laplace exponent of $\sigma^f$ with L\'evy measure $\nu$ and \begin{align}
f(\partial_t)q(t) \, = \, \mathcal{L}^{-1} \left[ f(\lambda) \mathcal{L}\left[ q\right](\lambda) - \lambda^{-1}f(\lambda) q(0) \right] (t).
\label{laplmeer}
\end{align}
If the subordinator considered is the $\beta$-stable subordinator then $f (\partial_t)$ reduces to the Dzerbayshan-Caputo fractional derivative. 

In \citet{kochu} the author sheds lights on the explicit form of the operator \eqref{laplmeer}. The author considered  derivatives of the form
\begin{align}
\mathbb{D}_{(k)}u(t) \, = \, \frac{d}{dt} \int_0^t k(t-\tau) u(\tau) d\tau -k(t) u(0)
\label{opkochuintro}
\end{align}
under suitable assumptions on the Laplace transform $\mathcal{L} \left[ k (\puntomio) \right]$ and he relates such operators with the relaxation and heat equations
\begin{align}
&\mathbb{D}_{(k)}u \, = \, -\lambda u, \label{relax} \\
& \mathbb{D}_{(k)}u \, = \, \Delta u \notag.
\end{align}
He pointed out that the solution to \eqref{relax} appears in the description given in \citet{meerpoisson} of the process $N(L^f(t))$, $t>0$, where $N$ is a Poisson process with rate $\lambda$ and $L^f$ is an inverse subordinator.
In \citet{toaldoconv} an approach based on operators of the form
\begin{align}
\mathcal{D}_t^f u \, = \, b\frac{d}{dt}u(t) + \frac{d}{dt} \int_0^t u(s) \bar{\nu} (t-s)ds,
\label{toar} \\
\mathfrak{D}_t^f u \, = \, b\frac{d}{dt}u(t) +  \int_0^t  u^\prime(s) \bar{\nu} (t-s)ds
\label{toac}
\end{align}
has been adopted. In \eqref{toar} and \eqref{toac}, $\bar{\nu}(s) = a + \nu (s, \infty)$ for a L\'evy triplet $(a, b,\nu)$ such that $f$ is the Bernstein function
\begin{align}
f(\lambda) \, = \, a+b\lambda + \int_0^\infty \l 1-e^{-\lambda s} \r \nu(ds).
\end{align}
Note that the Laplace symbol of the operator \eqref{toac} coincides with \eqref{laplmeer} and thus \eqref{toac} is a way to write $f\l \partial_t \r$.

\section{Distributed order subordinators}
All throughout the paper we will deal with Bernstein functions (\citet{artbern}). A Bernstein function is a $C^\infty$ non-negative function $f:(0, \infty) \to \mathbb{R}$ for which
\begin{align}
(-1)^kf^{(k)}(\lambda) \leq 0, \textrm{ for all } \lambda > 0 \textrm{ and } k \in \mathbb{N},
\end{align}
and has the represention
\begin{align}
f(\lambda) \, = \, a+b\lambda+\int_0^\infty \l 1-e^{-\lambda s} \r \nu(ds), \qquad a \geq 0, b \geq 0.
\end{align}
The measure $\nu$ is a non-negative $\sigma$-finite measure supported on $(0, \infty)$ such that the integrability condition
\begin{align}
\int_0^\infty \l s \wedge 1 \r \nu(ds) < \infty
\end{align}
is fulfilled (see more on Bernstein functions in \citet{librobern}). In what follows we will denote by $\mathrm{BF}$ the family of all Bernstein functions. In this work we will deal with functions of the form $f: (0, \infty) \times  E \to \mathbb{R}$ such that for all fixed $y \in E$, $\lambda \to f(\lambda, y)$ is a Bernstein function, and $E$ is a Polish space. Assume therefore that
\begin{align}
f(\lambda, y) \, = \, a(y) + b(y)\lambda + \int_0^\infty \l 1-e^{-s\lambda} \r \nu(ds,y), \qquad y \in E,
\label{bern2var}
\end{align}
where $a(y)$ and $b(y)$ are non-negative and $\nu(ds, y)$ is a non-negative $\sigma$-finite measure on $(0, \infty)$ such that
\begin{align}
\int_0^\infty \l s \wedge 1 \r \nu(ds, y) \, = \, V(y) < \infty , \, \forall y \in E.
\label{assump}
\end{align}
In what follows we want that for a probability measure $p$ on $W \subseteq E$ the function
\begin{align}
\lambda \to \int_W f(\lambda, y) p(dy)  =  \int_W \l a(y) + b(y) \lambda \r p(dy) + \int_0^\infty \l 1-e^{-\lambda s} \r \int_W \nu(ds, y) p(dy)
\end{align}
is a Bernstein function and therefore we will work all throughout the paper under the following assumptions
\begin{enumerate}
\item[\textbf{A1)}] \label{a1} Assume that the Borel-measurable functions $y \to a(y)$ and $y \to b(y)$ are bounded or that the probability measure $p$ is such that
\begin{align}
\int_W a(y) p(dy) < \infty \textrm{ and } \int_W b(y) p(dy) < \infty.
\end{align}
\item[\textbf{A2)}] \label{a2} Assume that the family $\ll \nu(ds, y) \rr_{y \in E}$ is such that for all $y \in E$ the function $V(y)$ defined in \eqref{assump} is bounded or that the probability measure $p$ is such that
\begin{align}
\int_W V(y) p(dy) < \infty.
\end{align}
\end{enumerate}
Note that under A2)
\begin{align}
\int_0^\infty (s\wedge 1) \int_W \nu(ds, y) p(dy) = \int_W V(y) p(dy) < \infty.
\end{align}

We will also need the function
\begin{align}
\frac{f(\lambda, y)}{\lambda} \, = \, b(y) + \int_0^\infty e^{-\lambda s} \, \bar{\nu}(s,y) \, ds, \qquad y \in E,
\label{berncode}
\end{align}
where
\begin{align}
\bar{\nu}(s,y) \, = \, a(y)+\nu((s, \infty), y), \qquad s>0, y \in E.
\end{align}
The representation \eqref{berncode} is obtained from \eqref{bern2var} by performing an integration by parts. Note that for all fixed $y \in E$ the function $\lambda \to \lambda^{-1}f(\lambda, y)$ is completely monotone i.e. it is $C^\infty$ and such that
\begin{align}
(-1)^k \frac{\partial^k}{\partial \lambda^k} \frac{f(\lambda, y)}{\lambda} \geq 0, \textrm{ for all } \lambda > 0 \textrm{ and } k \in \mathbb{N} \cup \ll 0 \rr.
\end{align}
Note also that under A1) and A2) it is true that for all $z>0$
\begin{align}
\int_z^\infty \int_W \nu(ds, y) p(dy) < \infty
\end{align}
since
\begin{align}
\int_W \nu(ds, y) p(dy)
\label{levmesmia}
\end{align}
is a L\'evy measure.

\subsection{Definition of the process}
In this section we introduce the concept of distributed order subordinator. This can be done by means of the L\'evy-It\^o decomposition. A classical subordinator $\sigma^f$ is a non-decreasing L\'evy process and it should be defined as (\citet{itodec})
\begin{align}
\sigma^f (t) \, = \, bt + \sum_{0 \leq s \leq t} \mathpzc{e}(s), \qquad b \geq 0,
\end{align}
where $\mathpzc{e}(s)$ is a Poisson point process with characteristic measure $a\delta_\infty+\nu(ds)$ (the reader can consult \citet{bertoinb, bertoins} for some details). Such a procedure is still valid if we consider a mixture approach. Let $\iota_y$ be a non-negative measure on $(0, \infty)$. Suppose $\iota_y (0, \infty)= \infty$, for all $y \in E$, but assume that there exists a partition of $(0, \infty)$, say $A_n$, into Borel sets such that the application $y \to c(y, n) = \iota_y(A_n)$ is bounded uniformly in $y$ or assume that the probability measure $p$ is such that
\begin{align}
\int_W \iota_y (A_n) p(dy) < \infty.
\end{align}
Consider now the product space $(0, \infty) \times [0, \infty)$, the measure $\int_W \iota_y (\puntomio) p(dy) \otimes \mathpzc{l}(dt)$ where $\mathpzc{l}(\puntomio)$ is the Lebesgue measure, and a Poisson measure $\psi$ on $(0, \infty) \times [0, \infty)$ such that for $B$ Borel
\begin{align}
\Pr \ll \psi \l B \times [s, t] \r =k\rr \, = \, \frac{\l (t-s)\int_W \iota_y(B)p(dy)\r^k}{k!} e^{-(t-s) \int_W \iota_y(B)p(dy)}.
\end{align}
Since $\psi \l B \times \ll t \rr \r = 0 \textrm{ or } 1$, we can write
\begin{align}
\psi (\puntomio)  = \sum_{t \geq 0} \delta_{[t,\mathpzc{e}(t)]} (\puntomio)
\end{align}
for $\mathpzc{e}(t)$ a Poisson point process with mixed characteristic measure $\int_W \iota_y (\puntomio) p(dy)$ and where $\delta_{[t,\mathpzc{e}(t)]}(\puntomio)$ is the Dirac point mass at the point $(t,\mathpzc{e}(t))$.
\begin{defin}
\label{defdos}
We define the distributed order subordinator according to the L\'evy-It\^o decomposition as
\begin{align}
\sigma^{f, p} (t) \, = \,  t\int_W b(y)p(dy) + \sum_{0 \leq s \leq t} \mathpzc{e}(s), 
\end{align}
where $\mathpzc{e}(s)$ is a Poisson point process, for $a(y)$, $b(y)$, $\nu$ and $p$ as in A1) and A2), with characteristic measure $\int_W \l \nu(dx, y) +a(y)\delta_\infty \r p(dy)$.
\end{defin}
Under A1) and A2) we may use Campbell theorem (see, for example, \citet{king}, page 28) for writing
\begin{align}
\mathds{E} e^{-\lambda \sigma^{f, p} (t) } \, = \, & \mathds{E} \exp \ll -\lambda  t\int_W b(y)p(dy) - \lambda \sum_{0\leq s \leq t} \mathpzc{e} (s) \rr \notag \\
= \, & \exp \ll - \lambda t \int_W b(y)p(dy) + \int_0^\infty \l e^{-\lambda s} -1 \r \mathds{E}  \psi \l ds \times [0, t] \r \rr \notag \\
= \, & \exp \ll -t \int_W f(\lambda, y) p(dy) \rr.
\end{align}

\begin{os}
As pointed out in \citet{bn} L\'evy mixing may be also explained by means of L\'evy bases. We recall here some basic facts. Let $Z \subset \mathbb{R}^k$ be a Borel set and let $\mathcal{B}_b(Z)$ be the bounded Borel sets of $Z$. A family $\ll \Lambda  (A): A \in \mathcal{B}_b(Z) \rr$ of r.v.'s is said to be a L\'evy basis on $Z$ if the following are true:
\begin{enumerate}
\item $\Lambda (A)$ is infinitely divisible for all $A \in \mathcal{B}_b(Z)$
\item If $A_1, \cdots, A_n$ are disjoint sets in $\mathcal{B}_b(Z)$ then the r.v.'s $\Lambda (A_1) \cdots \Lambda (A_n)$ are independent
\item If $A_1$, $A_2$, $\cdots$, are disjoint sets in $\mathcal{B}_b(Z)$ such that $\cup_{i=1}^\infty A_i \in \mathcal{B}_b(Z)$, then it is true that
\begin{align}
\Lambda \l \bigcup_{i=1}^\infty A_i \r \, = \, \sum_{i=1}^\infty  \Lambda (A_i),
\end{align}
where the right hand side converges a.s..
\end{enumerate}
Under the first assumption one can easily compute the cumulant transform of $\Lambda$ i.e. the transformation
\begin{align}
C \l \xi \ddagger \Lambda (A)   \r \, := \, &  \log \mathds{E}e^{i\xi \Lambda(A)} \notag \\
= \, & i \xi b(A) - \frac{1}{2} \xi^2 g(A) + \int_{\mathbb{R}} \l e^{i\xi s} -1 -i\xi s \mathds{1}_{[-1,1]}(s) \r U(ds, A),
\label{216}
\end{align}
where the measure $U(ds, A)$ is a L\'evy measure (on $\mathbb{R}$) for all fixed $A$ and a measure on $\mathcal{B}_b(Z)$ for all fixed $ds$. Such measure $U$ is said to be a generalized L\'evy measure and without loss of generality (see \citet{raja}) one can assume that there exists a measure $p$ such that $U(ds, dz) = \nu(ds, z) p(dz)$. In order to find the L\'evy processes examined in the present paper it is sufficient to extend the measure $U$ to the space $\l \mathbb{R} \times Z \times (0, \infty) \r$ as the product measure $U(ds, dz) \otimes \mathpzc{l}(dt)$ where $\mathpzc{l}(\puntomio)$ represents the Lebesgue measure on $(0, \infty)$. Then assume that $\Lambda$ is such that \eqref{216} becomes a Bernstein function (therefore $U(\puntomio, dz)$ must be supported on $(0, \infty)$) and that $p(Z)=1$ in such a way that 
\begin{align}
C \l \xi \ddagger \Lambda(dz \times [0,t]) \r \, = \,  \int_0^\infty \l e^{i\xi s}-1 \r \l \mathpzc{l}[0,t] \nu(ds, z) p(dz) \r
\end{align}
and by integrating over $Z$ the L\'evy measure becomes
\begin{align}
 t \int_Z \nu(ds, z) p(dz) 
\end{align}
which is a mixture of the form \eqref{levmesmia}.
\end{os}

\subsection{Distributional properties of the process}
The process $\sigma^{f, p}(t)$, $t \geq 0$, is a subordinator generated by the L\'evy triplet 
\begin{align}
\l \int_W a(y)p(dy), \int_W b(y)p(dy), \int_W\nu(ds, y)p(dy) \r, \qquad W \subseteq E,
\label{triplet}
\end{align}
and therefore has Laplace exponent
\begin{align}
\mathds{E} e^{-\lambda \sigma^{f, p}(t)} \, = \, \exp \ll -t \l \mathds{E} a(Y)+\mathds{E} b(Y)\lambda + \int_0^\infty \l 1-e^{-\lambda s} \r \mathds{E} \nu(ds, Y) \r \rr,
\end{align}
where $Y$ is a r.v. with law $p$ on $W \subseteq E$.
Here we study the distributional properties of distributed order subordinators, denoted by $\sigma^{f, p}(t)$, $t \geq 0$, with L\'evy triplet \eqref{triplet}, by making assumptions on the subordinator $\sigma^{f, y}$ corresponding to the L\'evy triplet $\l a(y), b(y), \nu(ds, y)\r$.

The transition probabilities of subordinators are convolution semigroups supported on $[0, \infty)$. A family $\mu_t$, $t\geq 0$, of sub-probability measures on $[0, \infty)$ is said to be a convolution semigroup if
\begin{enumerate}
\item $\mu_t[0, \infty) \leq 1$, for all $t \geq 0$,
\item $\mu_t * \mu_s = \mu_{t+s}$, for all $s, t \geq 0$,
\item $\mu_t \to \delta_0$, vaguely as $t \to 0$.
\end{enumerate}
The explicit form of the transition probabilities of subordinators is known just in some particular cases. However since such distributions are characterized by the Laplace transform
\begin{align}
\mathds{E} e^{-t \sigma^f (t)} \, = \, e^{-tf(\lambda)}
\end{align}
something could be said by observing at the Laplace exponent $f$. There exist indeed some special classes of subordinators with nice properties and such classes can be very often distinguished by observing the Laplace exponent $f$. This requires some theory of Bernstein functions. For a background on such a theory and on the relationships between family of Bernstein functions and convolution semigroups of sub-probability measures the reader can consult \citet{librobern} and the references therein.

In what follows we will always refer to $\sigma^{f, p}(t)$, $t>0$, as the distributed order subordinator obtained by randomizing the L\'evy triplet $(a(y), b(y), \nu(ds, y))$ associated to the subordinator $\sigma^{f, y}$.

\subsubsection{Bondesson class}
One of the most important class of sub-probability measures related to subordinators is the Bondesson class. This class of measures is closed under convolution and vague convergence (see Lemma 9.2. of \cite{librobern}) and is composed of infinitely divisible measures. A sub-probability measure $\mu$ is said to be of the Bondesson class, and we write $\mu \in \mathrm{BO}$, if
\begin{align}
\mathcal{L} \left[ \mu \right] (\lambda) \, = \, e^{-f(\lambda)}
\end{align}
where $f$ is a \emph{complete} Bernstein function. A Bernstein function $f$ is said to be complete, and we write $f \in \mathrm{CBF}$, if it has the representation
\begin{align}
f(\lambda) \, = \, a + b \lambda + \int_0^\infty \l 1-e^{-\lambda s} \r \mathfrak{m}(s)ds,
\label{35}
\end{align}
where the density $s \to \mathfrak{m}(s)$ is a completely monotone function.
\begin{prop}
Let $\mu_t^y(B) = \Pr \ll \sigma^{f,y}(t) \in B \rr$, $t\geq 0$. Assume that $\mu_t^y \in \mathrm{BO}$, for all $y \in E$.
Let $\mu_t^p (B) = \Pr \ll \sigma^{f, p} (t) \in B \rr$, where $\sigma^{f, p}$ is the corresponding distributed order subordinator. We have that $\mu_t^p \in \mathrm{BO}.$
\end{prop}
\begin{proof}
The fact that $\mu_t^y \in \mathrm{BO}$ is equivalent to saying that the Laplace exponent $\lambda \to f(\lambda, y)$ is a complete Bernstein function for all fixed $y$ and therefore  from \eqref{35}
\begin{align}
f(\lambda, y) \, = \, & a(y)+b(y)\lambda + \int_0^\infty \l 1-e^{-\lambda s} \r \nu(ds, y) \notag \\
= \, & a(y) + b(y) \lambda + \int_0^\infty \l 1-e^{-\lambda s} \r \mathfrak{m}_y(s) ds  , \qquad \textrm{for all } y \in E.
\end{align}
The function
\begin{align}
\mathds{E} f(\lambda, Y) \, = \, \mathds{E} a(Y) + \mathds{E} b(Y) \lambda + \int_0^\infty \l 1-e^{-\lambda s} \r \mathds{E} \mathfrak{m}_Y(s)ds
\end{align}
is again a complete Bernstein function. This can be ascertained by using the fact that since $\mathfrak{m}$ is completely monotone then it is the Laplace transform of a certain measure $m_y(dt)$ and thus
\begin{align}
\mathds{E} \mathfrak{m}_Y(s) \, = \,  \int_0^\infty e^{-st} \int_W  m_y(dt) p(dy)
\end{align}
is a completely monotone function representing the density of the L\'evy measure $\mathds{E} \nu(ds, Y)$.
\end{proof}
\subsubsection{Mixture of exponential distributions}
A measure $\mu$ on $[0, \infty)$ is said to be a mixture of exponential distributions, and we write $\mu \in \mathrm{ME}$, if
\begin{align}
\mu [0,t] \, = \, \int_{(0, \infty]} \l 1-e^{-t \alpha} \r \rho(d\alpha)
\label{condme}
\end{align}
for some sub-probability measure $\rho$ on $(0, \infty]$.
We recall that $\mathrm{ME} \subset \mathrm{BO}$ (see \cite{librobern} page 81 for further details) since it consists of the family of measures such that $\mathcal{L}[\mu] (\lambda) = e^{-f} = \frac{1}{g}$ where $1/g$ is a Stieltjies function (we write $1/g \in \mathrm{S}$), with $1/g(0+) \leq 1$ for $f \in \mathrm{CBF}$. A Stieltjes function is a function $h:(0, \infty) \to [0, \infty)$ with representation
\begin{align}
h(\lambda) \, = \, \frac{a}{\lambda} + b + \int_0^\infty \frac{1}{\lambda +t} w(dt), \qquad a, b \geq 0,
\end{align}
where $w(dt)$ is a measure on $(0, \infty)$ such that
\begin{align}
\int_0^\infty \frac{1}{1+t} w(dt) < \infty.
\end{align} 
The reader can also consult \citet{steuten}, Chapter VI, for detailed information on mixture of exponential distributions.
Formally a measure $\mu$ is said to be a mixture of exponential distributions if it is such that $\mathcal{L}\left[ \mu \right] (\lambda) \in \mathrm{S}$ with $\mathcal{L}\left[ \mu \right] (\lambda) \big|_{\lambda = 0} \leq 1$.
\begin{prop}
Let $\mu_t^y(dx) = \Pr \ll \sigma^{f, y}(t) \in dx \rr$ be such that $\mu_t^y \in \mathrm{ME}$, for all fixed $y \in E$. We have that the corresponding distributed order subordinator $\sigma^{f, p}(t)$, $t \geq 0$, is such that $\Pr \ll \sigma^{f, p}(t) \in dx \rr=\mu_t^p(dx) \in \mathrm{ME}$.
\end{prop}
\begin{proof}
From Theorem 9.5 of \cite{librobern} we know that the condition \eqref{condme} is equivalent to saying that there exists a function $\eta : (0, \infty) \to [0,1]$ satisfying $\int_0^1 \eta(t) t^{-1} dt< \infty$ and such that, for some $\beta \geq 0$,
\begin{align}
\mathcal{L} \left[ \mu \right] (\lambda) \, = \, & \exp \ll -\beta - \int_0^\infty \l \frac{1}{t} - \frac{1}{\lambda +t} \r \eta(t) dt \rr \notag \\
= \, & \exp \ll - \beta - \int_0^\infty \l 1-e^{-\lambda t} \r \nu(dt) \rr
\label{324}
\end{align}
for some L\'evy measure $\nu$.
Here we must have for all $y \in E$
\begin{align}
\mathcal{L} \left[ \mu_t^y \right] (\lambda) \, = \, \exp \ll -t \l a(y)+ \int_0^\infty \l 1-e^{-\lambda s} \r \nu(ds, y) \r \rr \in \mathrm{S}
\end{align}
and
\begin{align}
f(\lambda, y) \, = \,  a(y)+ \int_0^\infty \l 1-e^{-\lambda s} \r \nu(ds, y) \in \mathrm{CBF} 
\end{align}
Since $f \in \mathrm{CBF}$ the L\'evy measure $\nu(ds, y)$ has a density $\mathfrak{m}_y(s)$ such that $s \to \mathfrak{m}_y(s)$ is completely monotone for all fixed $y \in E$. This implies that $\mathfrak{m}_y(s)$ is the Laplace transform of a certain measure $m_{y}(dt)$, i.e. $\mathfrak{m}_{y}(s) = \int_0^\infty e^{-st} m_{y}(dt)$. 
Thus we can write
\begin{align}
 f(\lambda, y)  \, = \, &   a(y) +  \int_0^\infty \l 1-e^{-\lambda s} \r   \nu(ds, y) \notag \\
= \, &  a(y)  +  \int_0^\infty \l 1-e^{-\lambda s} \r  \int_0^\infty e^{-st} m_y(dt) ds \notag \\
= \, &  a(y) +   \int_0^\infty \int_0^\infty \l e^{-st}-e^{-s(\lambda +t)} \r ds \, m_y(dt)  \notag \\
= \, &  a(y) +   \int_0^\infty  \l \frac{1}{t}- \frac{1}{\lambda + t} \r  \, m_y(dt).
\label{calcul}
\end{align}
Since $\mu_t^y \in \mathrm{ME}$, for all $y \in E$ we must have that $m_y(dt)$ has a density $\eta_y(t)$ such that
\begin{equation}
\eta_y: (0, \infty) \to [0,1] \textrm{ with } \int_0^1 w^{-1} \eta_y(w)dw < \infty,
\label{corro}
\end{equation}
and this permits us to conclude
\begin{align}
\mathcal{L} \left[ \mu_t^p \right] (\lambda)  \, = \, \exp \ll - \mathds{E} a(Y) - \int_0^\infty  \l \frac{1}{t}- \frac{1}{\lambda + t} \r  \, \mathds{E} \eta_Y(t) dt \rr.
\end{align}
The fact that
\begin{align}
\mathds{E}  \eta_Y : (0, \infty) \to [0,1]
\end{align}
is verified since
\begin{align}
0 \leq \int_W \eta_y (\puntomio) p(dy) \leq \sup_{y \in W} \eta_y (\puntomio),  \textrm{ with }  \sup_{y \in W} \eta_y \in [0,1].
\end{align}
Furthemore note that under A2)
\begin{align}
\infty \, > \, & \int_0^\infty (s\wedge 1) \int_W \nu(ds, y) p(dy) \notag \\
 = \, & \int_W \int_0^\infty (s \wedge 1) \int_0^\infty e^{-st} \eta_y(t)\,dt \, ds \, p(dy) \notag \\
 = \, & \int_W \int_0^\infty \frac{1}{t^2} \l 1-e^{-t} \r \eta_y(t)\,dt \, p(dy) \notag \\
 \geq \, & \int_W \int_0^1\frac{1}{t^2} \l 1-e^{-t} \r \eta_y(t)\,dt \, p(dy) \notag \\
 \geq \, & \int_W \int_0^1\frac{1}{t} \eta_y(t)\,dt \, p(dy)
\end{align}
and therefore we have proved that
\begin{align}
\int_W \int_0^1\frac{1}{t} \eta_y(t)\,dt \, p(dy)  < \infty.
\end{align}
\end{proof}

\subsubsection{Generalized gamma Convolution}
A measure $\mu$ on $[0, \infty)$ is said to be a Generalized gamma convolution, and we write $\mu \in \mathrm{GGC}$ if
\begin{align}
\mathcal{L} \left[ \mu \right] (\lambda) \, = \, e^{-f(\lambda)}, \textrm{ where $f$ is a Thorin-Bernstein function.}
\end{align}
A quite important sublcass of Bernstein functions consists in those Bernstein functions whose derivative is a Stieltjes function. Such functions are called Thorin-Bernstein functions (we write $f \in \mathrm{TBF}$) and are of the form
\begin{align}
f(\lambda) \, = \, a + b \lambda + \int_0^\infty \l 1-e^{-\lambda t} \r \, \tau(t) dt
\end{align}
where $\int_0^\infty (t \wedge 1) \tau(t) dt < \infty$ and $t \to t \, \tau(t)$ is a completely monotone function. The class of $\mathrm{GGC}$ is the smallest class of sub-probability measures on $[0, \infty)$ closed under convolution and vague limits containing all gamma distributions (see Theorem 9.13 of \cite{librobern}).

\begin{prop}
Let $\Pr \ll \sigma^{f, y}(t) \in \puntomio \rr \, = \, \mu_t^y (\puntomio) \in \mathrm{GGC}$. For the distributed order subordinator $\sigma^{f, p}$ we have that
\begin{align}
\Pr \ll \sigma^{f, p} (t) \in \puntomio \rr \, = \,  \mu_t^p (\puntomio) \in \mathrm{GGC}.
\end{align}
\end{prop}
\begin{proof}
Since $\mu_t^y \in \mathrm{GGC}$ we have that
$\lambda \to f(\lambda, y) \in \mathrm{TBF}$ for all $y \in E$ and thus
\begin{align}
f(\lambda, y) \, = \, a(y) + b(y) \lambda + \int_0^\infty \l 1-e^{-\lambda t} \r \, \tau_y(t) \, dt.
\end{align}
It is sufficient to prove that
\begin{align}
\mathds{E} f(\lambda, Y) \, = \,  \mathds{E} a(Y) + \mathds{E} b(Y) \lambda + \int_0^\infty \l 1-e^{-\lambda t} \r \, \mathds{E} \tau_Y(t) \, dt \in \mathrm{TBF}.
\end{align}
Note that since $\lambda \to f(\lambda, y) \in \mathrm{TBF}$, one has
\begin{equation}
t \to t \, \tau_y(t) \, = \, \int_0^\infty e^{-ts} \phi_y(ds), \textrm{ for all } y \in E,
\end{equation}
for a certain measure $\phi_y(ds)$. Thus we can write
\begin{align}
t \to t \, \mathds{E}  \tau_Y(t) \, = \, & t \int_W \tau_y(t) \, p(dy) \notag \\
= \, & \int_W \int_0^\infty e^{-ts} \phi_y(ds) \, p(dy) \notag \\
= \, & \int_0^\infty e^{-ts} \, \mathds{E} \phi_Y(ds),
\end{align}
which is a completely monotone function.
\end{proof}

\section{Special distributed order subordinators and inverse local times}
In this section we point out the relationships between special subordinators, distributed order subordinators and inverse local times of Markov processes.

A subordinator with Laplace exponent $f$ is said to be special if $\lambda/f(\lambda) \in \mathrm{BF}$. This is because a Bernstein function $f$ is said to be special if $\lambda /f(\lambda) \in \mathrm{BF}$. The collection of all special Bernstein functions will be denoted by $\mathrm{SBF}$. We recall that in case $f \in \mathrm{SBF}$ we clearly have that $f^\star = \frac{\lambda}{f} \in \mathrm{SBF}$. We call $f^\star = \lambda /f$ the conjugate of $f$.
The most important property of special subordinators concerns their potential measures
\begin{align}
U^f (dx) \, = \,  \mathds{E}  \int_0^\infty \mathds{1}_{\ll \sigma^f(t) \in dx \rr} dt
\end{align}
which is such that
\begin{align}
\mathcal{L} \left[ U^f(dx) \right] (\lambda) \, = \, \frac{1}{f(\lambda)}.
\label{laplpotmeas}
\end{align}
It is well-known that $\sigma^f$ is special if and only if
\begin{align}
U^f (dx) \, = \, c \delta_0(dx) + u(x) dx
\end{align}
for $c \geq 0$ and for some non-increasing function $u:(0, \infty) \to (0, \infty)$ satisfying $\int_0^1 u(x) dx< \infty$.
The reader can consult \citet{librobern}, Chapter 10, or \citet{song}, for further information on special Bernstein functions and the related subordinators.

In the upcoming theorem we will need the following information on local times.
Let $X = \l \Omega, \mathpzc{M}, \mathpzc{M}_t, X_t, \theta_t, P^z \r$ be a temporally homogeneous Markov process on the Polish space $E$. As usual given $\l E, \mathpzc{E} \r$ we have $E_\Delta = E \cup \ll \Delta \rr$ and $\mathpzc{E}_\Delta$ is the $\sigma$-algebra in $E_\Delta$ where $\Delta$ is clearly a point not in $E$ for which for all $t \in [0, \infty]$, $X_t(\omega) = \Delta$ implies $X_s(\omega) = \Delta$, for all $s \geq t$. $\l \Omega, \mathpzc{M} \r$ is a measurable space and $\mathpzc{M}_t$ an increasing family $\ll \mathpzc{M}_t : t \in [0, \infty] \rr$. With $\theta_t$ we denote the translation operator i.e. the map $\theta_t: \Omega \to \Omega$ such that $\theta_\infty \omega = \omega_\Delta$, for all $\omega$ and where $\omega_\Delta$ is a distinguished point of $\Omega$. Due to homogeneity we write $X_t \circ \theta_h = X_{t+h}$, for all $t, h \in [0, \infty]$. Let
\begin{align}
T_tu \, = \, \int_E P_t(z, dy) \, u(dy) \, = \, \mathds{E}^z u(X_t)
\end{align}
be the semigroup associated with $X_t$ and let
\begin{align}
R^\lambda u \, = \, \mathds{E}^z \l \int_0^\infty e^{-\lambda t} u(X_t) dt \r \, = \, \int_0^\infty e^{-\lambda t} T_tu \, dt
\end{align}
be the corresponding $\lambda$-potential operator. In what follows we suppose that $X$ is in duality with $\widehat{X}$, i.e.
\begin{align}
& R^\lambda u(z) \, = \, \mathds{E}^z \l \int_0^\infty e^{-\lambda t} u(X_t) dt \r \, = \, \int_E r^\lambda (z, x) u(x) \xi(dx) \label{nohat} \\
& \widehat{R}^\lambda u(z) \, = \, \widehat{\mathds{E} }^z \l \int_0^\infty e^{-\lambda t} u(\widehat{X}_t) dt \r \, = \, \int_E r^\lambda (x, z) u(x) \xi(dx)
\label{hat}
\end{align}
for some $\sigma$-finite measure $\xi$ on $E$, where $\widehat{X}= \l \Omega, \widehat{\mathpzc{M}}, \widehat{\mathpzc{M}_t}, \widehat{X}_t, \widehat{\theta}_t, \widehat{P}^z   \r$ is a temporally homogeneous Markov process on $E$. See \citet{b-g, blumget} for details on duality and potential theory. This implies that the measure on $E$
\begin{align}
U^\lambda (dy) \, = \, \int_0^\infty e^{-\lambda t} P_t(z, dy)
\end{align}
is absolutely continuous with respect to $\xi$ with density $r^\lambda$, i.e.
\begin{align}
U^{\lambda} (dy) \, = \, r^\lambda (z, y) \xi (dy).
\end{align}
The measure $\xi$ is said to be the reference measure.
Consider now a family of $\mathpzc{M}_t$-measurable r.v.'s $\ll A_t, t \geq 0 \rr$ on $\l \Omega, \mathpzc{M}, \mathpzc{M}_t \r$ such that the map $t \to A_t$ is a.s. continuous and non decreasing with $A_0=0$ and satisfies $A_{t+s}= A_t + A_s \circ \theta_t$ for all $s,t$. Let $R_A(\omega)= \inf \ll t : A_t(\omega)>0 \rr$; the family $\ll A_t, t \geq 0 \rr$ is said to be a local time at $y$ of $X$ if $P^y(R_A =0)=1$ and for all $x \neq y$, $P^y(R_A =0)=0$.

A necessary and sufficient condition for the existence of the local time in some $y \in E$ of a Markov process is that $y$ is regular for itself (see \cite{b-g} Theorem 3.13) i.e.
\begin{align}
P^{y} \l T_{y} = 0 \r = 1, \textrm{ where } T_{y} \, = \, \inf \ll t \geq 0 : X(t) = y \rr.
\end{align}
However in the above framework Proposition 7.3 of \cite{blumget} holds and thus one has the following equivalent conditions for a point $y \in E$ to be regular for itself
\begin{enumerate}
\item $y$ is regular for $y$
\label{1}
\item $r^\lambda (z, y) \leq r^\lambda (y, y)< \infty$, for all $z \in E$
\label{2}
\item The function $z \to r^\lambda (z, y)$ is bounded and continuous at $z=y$
\label{3}
\end{enumerate}
Assume that the function $z \to r^\lambda(z, y)$ is lower semi-continuous for all $y \in E$. Denote by $L(y, t)$ the local time at $y$. The random measure $dL_t$ supported on $\ll t : X_t = y \rr$ is such that
\begin{align}
r^\lambda (x, y) \, = \, \mathds{E}^x \int_0^\infty e^{-\lambda t} dL(y, t)
\label{fumaparei}
\end{align}
and it is true that, a.s.,
\begin{equation}
L (y, t) \, = \, \lim_{\epsilon \to 0} \int_0^t \delta_{\epsilon, y}^\xi(X_s) ds.
\label{deltastrana}
\end{equation}
With the symbol $\delta_{\epsilon, y}^\xi(x)$ we denote a family of continuous functions on $E$ with compact support $S_\epsilon$ in a neighborhood of $y$ such that $\lim_{\epsilon \to 0} S_\epsilon = \ll y \rr$ and for which
\begin{equation}
\int \delta_{\epsilon, y}^\xi (x) \xi(dx) = 1.
\end{equation}

It is well-known that the inverse local time at $y$ of a Markov process is a subordinator  (see, for example, \citet{bertoinb}, Chapter IV, Theorem 8, or \citet{b-g}, page 219), provided that such a local time exists. 
In the framework introduced above it is well known that the inverse local time at $y$ of $X_t$, say
\begin{align}
L^{-1}(y, t) \, = \, \inf \ll s>0 : L(y, s) >t \rr,
\label{definvlot}
\end{align}
exists and is a subordinator with Laplace exponent $1/r^\lambda(y, y)$.

\begin{te}
Let $Y$ be a random variable on $W \subseteq E$ with law $p$. Let 
\begin{align}
\lambda \to f(\lambda, y) \, = \, a(y) + b(y) \lambda +\int_0^\infty \l 1-e^{-\lambda s} \r \nu(ds, y)
\end{align}
be a special Bernstein function for all fixed $y \in E$ and let $f^\star = \lambda /f$ be the conjugate of $f$. Assume that the L\'evy measure of $f^\star$, say $\nu^\star (\puntomio)$, and the probability measure $p$ are related by A2). We have the following results.
\begin{itemize}
\item[i)] The function $\breve{f^\star} (\lambda)\, = \, \mathds{E} f^\star (\lambda, Y)$ is a special Bernstein function and has representation
\begin{align}
\breve{f^\star}(\lambda) \, = \, \mathds{E} f^\star (\lambda, Y)\, = \, \mathds{E} a^\star(Y) + \mathds{E} b^\star(Y)\lambda + \int_0^\infty \l 1-e^{-\lambda s} \r \mathds{E}  \nu^\star (ds, Y),
\end{align}
where
\begin{align}
b^\star (y) \, = \,
\begin{cases}
0, \qquad & b(y) > 0, \\
\frac{1}{a(y) + \nu((0, \infty), y)}, & b(y) =0,
\end{cases},
\; a^\star (y) \, = \, \begin{cases}
0, \qquad &a(y) >0, \\
\frac{1}{b(y) + \int_0^\infty t \nu(dt, y)}, & a(y) = 0,
\end{cases}
\label{312}
\end{align}
provided that $p$ is such that A1) is fulfilled for the function $a^\star(y)$ and $b^\star(y)$.

\item[ii)] Suppose that $f(\puntomio, y)$ is the Laplace exponent of the inverse local time at $y$ of $X_t$, where $X_t$ is the process above. Assume that one of the conditions \eqref{1}, \eqref{2}, \eqref{3} holds for all $y \in W$. Let $P^p = \int_W P^y p(dy)$. Let $\bar{X} = \l \Omega , \bar{\mathpzc{M}}, \bar{\mathpzc{M}}_t, \bar{X}_t, \bar{\theta}_t, P^p \r$ be the Markov process on $E$ with initial distribution $p$ such that $\bar{X}$ is identical in law at $X$ under $p(dy)= \delta_y$, for all $y \in W$. The conjugate of $\breve{f^\star}$, i.e. the function $\breve{f^\star}^\star (\lambda) = \frac{\lambda}{\mathds{E} f^\star (\lambda, Y)}$, is the Laplace exponent of the inverse local time of $\bar{X}_t$ at its random starting point $\bar{X}_0$.

\item[iii)] Denote the inverse local time of $\bar{X}_t$ at $\bar{X}_0$ by $L^{-1}(\bar{X}_0, t)$ and by $U^{f,p}(dt)$ its potential measure. Let $U^{f,y}(dt)$ be the potential measure of the inverse local time at $y$ of $X_t$. We have that
\begin{align}
U^{f,p}(dt) \, = \,  & \int_W U^{f, y}(dt) \, p(dy) \notag \\
= \, &\int_W \l b^\star(y) \delta_0(dt) + \bar{\nu^\star} (t, y) \, dt \r p(dy)
\end{align}
where $\bar{\nu^\star}(s, y) \, = \, a^\star(y) + \nu^\star((t, \infty),y)$.

\end{itemize}
\end{te}
\begin{proof}
\begin{enumerate}
\item[i)] Since $f (\lambda, y) \in \mathrm{SBF}$ for all fixed $y \in E$ we have that $f^\star(\lambda, y) = \lambda/f(\lambda, y) \in \mathrm{BF}$. In particular by applying Theorem 10.3 of \cite{librobern} we may write
\begin{align}
f^\star (\lambda, y) \, = \, a^\star(y) + b^\star(y) \lambda + \int_0^\infty \l 1-e^{-\lambda s} \r \nu^\star(ds,y)
\end{align}
where $a^\star (y)$ and $b^\star (y)$ are defined in \eqref{312}.
Let $\sigma^{f,y}$ be the subordinator with Laplace exponent $f(\lambda, y)$. Since $\lambda \to f(\lambda, y) \in \mathrm{SBF}$ we have that the corresponding subordinator $\sigma^{f, y}$ has potential measure (parametrized by $y$)
\begin{align}
U^{f, y} (dt) \, = \, b^\star(y) \delta_0(dt) + u_y(t) dt
\label{nonsene}
\end{align}
for a function $t \to u_y(t) : (0, \infty) \to (0, \infty)$ non-increasing and satisfying
\begin{align}
\int_0^1 u_y(t) dt < \infty, \qquad \textrm{for all } y \in E.
\label{4119}
\end{align}
In particular we deduce from Theorem 10.3, page 94, of \cite{librobern} that
\begin{align}
u_y(t) \,= \, \bar{\nu^\star}(t, y) \, = \, \nu^\star((t, \infty), y) +a^\star (y).
\end{align}
Now consider the measure
\begin{align}
\int_W U^{f, y}(dt) p(dy) \, = \, \mathds{E}  b^\star(Y) \delta_0(dt) + \mathds{E}  \bar{\nu^\star} (t, Y) \, dt.
\label{consmeas}
\end{align}
Observe that under $A1)$ and $A2)$ \eqref{consmeas} is well defined since the integrals converge.
The function $t \to \mathds{E}  \bar{\nu^\star} (t, y):(0, \infty) \to (0, \infty)$ is non-increasing and (under A1) and A2)) satisfies \eqref{4119} and thus \eqref{consmeas} is the potential measure of some special subordinator. In particular since
\begin{align}
\mathcal{L} \left[ \int_W U^{f, y}(dt) p(dy) \right] (\lambda) \, = \, \lambda^{-1}\mathds{E}  f^\star (\lambda, Y) \, = \, \frac{1}{\frac{\lambda}{\mathds{E} f^\star (\lambda, Y)}}
\end{align}
we use \eqref{laplpotmeas} and we deduce that $\frac{\lambda}{\mathds{E} f^\star (\lambda, Y)} \in \mathrm{SBF}$ which implies that $\mathds{E} f^\star (\lambda, Y) \in \mathrm{SBF}$.
\item[ii)] Now we consider the inverse local time $L^{-1}(\bar{X}_0, t)$. Since we assume that one of the conditions \eqref{1}, \eqref{2}, \eqref{3} holds for all $y \in W \subseteq E$ we have that the points $y \in W$ are regular since
\begin{align}
P^p \ll T_{\bar{X}_0} = 0 \rr \, = \, 1\textrm{ where } T_{\bar{X}_0} \, = \, \inf \ll t \geq 0 : \bar{X}_t = \bar{X}_0 \rr,
\end{align}
and this guarantees that such a local time exists. We recall that
\begin{align}
\mathds{E}^y \int_0^\infty e^{-\lambda L^{-1}(y, t)} dt\, = \, \mathds{E}^y \int_0^\infty e^{-\lambda t }dL(y, t) \, \stackrel{\eqref{fumaparei}}{=} \, r^\lambda (y, y).
 \label{417}
\end{align}
By Proposition 2.3 in chapter V of \cite{b-g}
\begin{align}
L^{-1}(y, t+s) \, = \, L^{-1}(y, t) + L^{-1}(y, s) \circ \theta_{L^{-1}(y, t)}
\label{statet}
\end{align}
and by using the definition \eqref{definvlot} of inverse local time we deduce that $\bar{X} \l L^{-1}(\bar{X}_0, t) \r = \bar{X}_0$ and therefore we perform calculations similar to that of Theorem 3.17 Chapter V of \cite{b-g} and we write
\begin{align}
  \mathds{E}^p e^{-\lambda L^{-1}(\bar{X}_0, t+s)} \, = \, & \mathds{E}^p e^{-\lambda L^{-1}(\bar{X}_0,t)} \mathds{E}^{\bar{X}\l L^{-1}(\bar{X}_0, t)\r} e^{-\lambda L^{-1}(\bar{X}_0,s)} \notag \\
 = \, & \mathds{E}^p e^{-\lambda L^{-1}(\bar{X}_0,t)} \mathds{E}^p e^{-\lambda L^{-1}(\bar{X}_0,s)}.
 \label{semig}
\end{align} 
Now in view of \eqref{fumaparei} and \eqref{deltastrana} we may write
\begin{align}
\mathds{E}^p \int_0^\infty e^{-\lambda L^{-1}(\bar{X}_0, t)} dt \, = \,&  \mathds{E}^p \int_0^\infty e^{-\lambda t} dL\l \bar{X}_0, t \r \notag \\
= \, & \int_W r^{\lambda}(y, y) p(dy)
\end{align}
and from \eqref{semig} we conclude that
\begin{align}
\mathds{E}^p e^{-\lambda L^{-1}(\bar{X}_0,t)} \, = \, e^{-t k} \textrm{ with } k = \frac{1}{\int_W r^{\lambda}(y, y) p(dy)}.
\label{421}
\end{align}
Now observe that since we suppose that $f(\lambda, y)$ is the Laplace exponent of the inverse local time of $X_t$ at $y$ we must have
\begin{align}
f(\lambda, y) \, = \, & a(y) + b(y) \lambda + \int_0^\infty \l 1-e^{-\lambda s} \r \nu(ds, y) \notag \\
= \, & \frac{1}{r^\lambda (y, y)}
\end{align}
and thus we get that
\begin{align}
\mathds{E}^p e^{-\lambda L^{-1}(\bar{X}_0, t)} \, = \, &  \exp \ll -t \frac{1}{\int_W \frac{1}{f(\lambda, y)}p(dy)}  \rr \notag \\
= \, &\exp \ll -t \frac{\lambda}{\int_W  f^\star (\lambda, y) p(dy)}  \rr \notag \\
= \, & \exp \ll -t \breve{f^\star}^\star (\lambda)  \rr.
\label{419}
\end{align}
Therefore
\begin{align}
\mathds{E}^p e^{-\lambda L^{-1}(\bar{X}_0,t)} \, = \, e^{-t \breve{f^\star}^\star (\lambda)}.
\end{align}
Note that the process $L^{-1}(\bar{X}_0, t)$ is clearly non-decreasing and has stationary independent increments in the ordinary sense if 
\begin{equation}
\lim_{\lambda \to 0} \int_W r^\lambda (y, y)p(dy) = \infty.
\label{killing}
\end{equation}
By adapting Proposition 3.19, chapter V, of \cite{b-g} we may prove that 
\begin{align}
& P^p \ll \bigcap_{j=1}^n \ll L^{-1}(\bar{X}_0, t_j) - L^{-1} \l \bar{X}_0, t_{j-1} \r \rr \in B_j ; L^{-1}(\bar{X}_0, t_n) < \infty \rr \notag \\
 = \, & \prod_{j=1}^n P^p \ll L^{-1}(\bar{X}_0, t_j - t_{j-1}) \in B_j \rr
\end{align}
for $B_j$, $j=1, \cdots, n$, disjoint sets and $0 = t_0 < \cdots < t_n$ by using \eqref{statet} and the fact that $\bar{X} \l L^{-1}\l \bar{X}_0, t \r \r = \bar{X}_0$ under $L^{-1}\l \bar{X}_0, t \r< \infty$. From \eqref{421} we know that this last condition is verified under \eqref{killing}. If \eqref{killing} holds we have
\begin{align}
\lim_{\lambda \to 0} \breve{f^\star}^\star (\lambda) \, = \, \lim_{\lambda \to 0} \frac{1}{\mathds{E} r^\lambda(Y,Y)} \, = \,  0.
\end{align}
If instead $\lim_{ \lambda \to 0} \int_W r^\lambda(y, y)p(dy) < \infty$ one has
\begin{align}
0 < \lim_{\lambda \to 0} \breve{f^\star}^\star (\lambda)  \, < \infty
\end{align}
and we get what in the literature is known as a killed subordinator. We can apply Theorem 3.21, chapter V, of \cite{b-g} and state that there exists a subordinator $\sigma(t)$ and an independent non-negative r.v. $\zeta$ on some $\l \Omega^\star, \Sigma, P \r$ with $P \ll \zeta \geq t \rr = \exp \ll -t\lim_{\lambda \to 0}  \frac{1}{\int_W r^\lambda (y, y)p(dy) } \rr$ such that
\begin{align}
\overline{L^{-1}}(\bar{X}_0, t) \, = \, \begin{cases}
\sigma(t), \qquad & t < \zeta, \\
\infty, & t \geq \zeta,
\end{cases}
\end{align}
and $L^{-1} \l \bar{X_0}, t \r$ are stochastically equivalent.
\item[iii)] In view of \eqref{nonsene} we can write
\begin{align}
U^{f,y}(dt) \, = \, b^\star(y) \delta_0(dt) + \bar{\nu^\star} (t, y) \, dt
\end{align}
and therefore
\begin{align}
\mathcal{L} \left[ \mathds{E} b^\star(Y) \delta_0(dt) + \mathds{E}  \bar{\nu^\star}(t, Y)dt \right] (\lambda) \, = \, \frac{\mathds{E} f^\star(\lambda, Y)}{\lambda} \, = \, \frac{1}{\frac{\lambda}{\mathds{E} f^\star(\lambda, Y)}}.
\end{align}
Now from \eqref{419} we can write
\begin{align}
\mathcal{L} \left[ \mathds{E}^p \int_0^\infty \mathds{1}_{\ll L^{-1}(\bar{X}_0, x) \in dt \rr} dx \right] (\lambda) \, = \, \frac{1}{\frac{\lambda}{\mathds{E} f^\star(\lambda, Y)}},
\end{align}
and the proof is complete.
\end{enumerate}
\end{proof}
In the previous Theorem the process $\bar{X}_t$ may be interpreted as the process $X_t$ started at the random point $Y$ in $W$. It is therefore clear that if the process $X_t$ is spatially homogeneous the theorem is still valid but the dependence on the starting point $y$ disappear. Heuristically this is due to the fact that in this case the local time behaves in the same way at all the regular points.

\section{Distributed order integro-differential operators and governing equations}
In this section we introduce distributed order integro-differential operators by means of which we write the governing equations of the distributed order subordinators. Furthermore we are interested in the governing equations of their inverse processes. We will consider the inverse of $\sigma^{f, p}(t)$, $t>0$, under the assumption $\nu(0, \infty)=\infty$, defined as
\begin{align}
L^{f, p} (t) \, = \, \inf \ll s \geq 0 : \sigma^{f, p} (s) > t  \rr
\label{definizioneinverso}
\end{align}
for which
\begin{align}
\ll L^{f, p} (t) > x \rr \, = \, \ll \sigma^{f, p} (x) < t \rr.
\label{proprietàinverso}
\end{align}

\subsection{Distributed order integro-differential operators}
According to Theorem 4.1 of \citet{meertri} we know that the density of the inverse of a subordinator is the fundamental solution in the mild sense of the pseudo-differential problem
\begin{align}
f \l \partial_t \r u(x, t) \, = \, -	\frac{\partial}{\partial x} u(x, t) +\delta(x) \nu(t, \infty),
\end{align}
where $f(\partial_t)$ is the following inverse Laplace transform
\begin{align}
f(\partial_t ) u \, = \, \mathcal{L}^{-1} \left[ f(\lambda) \mathcal{L} \left[ u   \right] (\lambda) - \lambda^{-1} f(\lambda) u(0) \right] (t).
\label{invlaplmeer}
\end{align}
According to \cite{meertri} a solution to a space-time pseudo-differential equation is said to be mild if its Fourier-Laplace of Laplace-Laplace transform solves the corresponding algebraic equation. Clearly if $f(\lambda) = \lambda^\alpha$ we retrive from \eqref{invlaplmeer} the Dzerbayshan-Caputo derivative.

In \citet{kochu} the author introduced operators of the form
\begin{align}
\mathbb{D}_{(k)}u \, = \, \frac{d}{dt} \int_0^t k(t-\tau) u(\tau) d\tau -k(t) u(0)
\label{derivatakochu}
\end{align}
under assumptions on the Laplace transform of $k(t)$ (see formula (1.6) of \cite{kochu}). Clearly if $k(t) = \frac{t^{-\alpha}}{\Gamma(1-\alpha)}$ one retrives from \eqref{derivatakochu} the regularized Riemann-Liouville fractional derivative having Laplace symbol $\mathcal{L} \left[ \mathbb{D}_{(k)}u \right] (\lambda) = \lambda^\alpha \widetilde{u}- \lambda^{\alpha -1} u(0)$. In the present work we follow the last approach (in a way similar to \cite{toaldoconv}) but we focus on the distributed order case. We introduce the following distributed order integro-differential operator.
\begin{defin}
\label{dev1}
Let $u$ be an absolutely continuous function on the interval $[c, d]$. Let $f$ be the Bernstein function as in \eqref{bern2var} and assume that the function $s \to \bar{\nu}(s,y)$ is absolutely continuous on $(0, \infty)$ for all $y \in E$. We define the generalized distributed order Riemann-Liouville derivative as
\begin{align}
_c\mathcal{D}_t^{f, p} u(t) \, = \, \mathds{E}  b(Y)\frac{d}{dt}u(t) + \frac{d}{dt}  \int_c^{t}   u( s) \mathds{E}  \bar{\nu}(t-s, Y) ds, \qquad c < t < d,
\label{defriemvar}
\end{align}
where
\begin{align}
\mathds{E}  \bar{\nu}(s, Y)  \, = \, \mathds{E}  \l a(Y) + \nu \l (s, \infty), Y \r \r,
\label{47}
\end{align}
and $Y$ is a r.v. with law $p$ on $W \subseteq E$.
If $\mathds{E} b(Y) = 0$ formula \eqref{defriemvar} makes sense for a continuous function $u$ as it happens for the classical distributed order fractional derivative (\cite{meerdo} formula (2.6) and (2.7)).
\end{defin}
By comparing Definition \ref{dev1} with the operator $f(\partial_t)$ (\cite{meertri}) defined as in \eqref{invlaplmeer} and with \eqref{toar}, \eqref{toac} (\cite{toaldoconv}) we note the following relationships.
\begin{lem}
\label{dafare}
We have that
\begin{align}
&\mathcal{L} \left[ _0\mathcal{D}_t^{f, p} u(t) \right] (\lambda) \, = \, \mathds{E}  f(\lambda, Y) \widetilde{u}(\lambda) - \mathds{E} b(Y) u(0) \label{primonellemma}  \\
&\mathcal{L} \left[ _0\mathcal{D}_t^{f, p} u(t) - \mathds{E} \bar{\nu}(t, Y) u(0) \right] (\lambda) \, = \, \mathds{E} f(\lambda, Y) \widetilde{u}(\lambda)- \lambda^{-1} \mathds{E} f(\lambda, Y) u(0) \label{secondonellemma}
\end{align}
\end{lem}
\begin{proof}
The results follow by explicitly computing the transformation as
\begin{align}
\mathcal{L} \left[ _0\mathcal{D}_t^{f, p} u(t) \right] (\lambda) \, = \, & \mathcal{L} \left[ \mathds{E}  b(Y)\frac{d}{dt}u(t) + \frac{d}{dt}  \int_0^{t}   u( s) \mathds{E}  \bar{\nu}(t-s, Y) ds \right] (\lambda) \notag \\
= \, & \mathds{E} b(Y) \l \lambda  \widetilde{u} (\lambda) - u(0) \r + \lambda \mathcal{L} \left[ u \, * \, \mathds{E} \bar{\nu}  \right] (\lambda) \notag \\
= \, & \mathds{E} b(Y) \l \lambda  \widetilde{u} (\lambda) - u(0) \r + \lambda  \l \widetilde{u}(\lambda) \l \lambda^{-1} \mathds{E}  f(\lambda, Y)- \mathds{E}  b (Y) \r \r \notag \\
= \, & \mathds{E} f(\lambda, Y) \widetilde{u}(\lambda) - \mathds{E} b(Y) u(0)
\label{bastaa}
\end{align}
where we used \eqref{berncode}, and this proves \eqref{primonellemma}. With the symbol $u * v$ we mean the Laplace convolution. The expression \eqref{secondonellemma} follows from \eqref{berncode} and \eqref{bastaa} since
\begin{align}
\mathcal{L} \left[ \mathds{E}  \bar{\nu}(s, Y) u(0) \right] (\lambda) \, = \, \l \mathds{E}   \lambda^{-1} f(\lambda, Y) - \mathds{E} b(Y) \r u(0).
\end{align}
\end{proof}
\begin{os}
\label{remfrac}
Note that if $\mathds{E}b(Y)=0$, and
\begin{align}
\nu(ds, y) \, = \, \frac{\alpha(y) s^{-\alpha(y) -1}}{\Gamma(1-\alpha(y))}ds
\end{align}
for a function $\alpha(y)$ strictly between zero and one and such that A2) is fulfilled, the integro-differential operator of Definition \ref{dev1} becomes
\begin{align}
_c\mathcal{D}_t^{f, p} u(t) \, = \, \frac{d}{dt }\int_c^{t} u(s) \int_W \frac{(t-s)^{-\alpha (y)}}{\Gamma(1-\alpha(y))}  p(dy) ds
\label{riep}
\end{align}
We obtain therefore a form of the distributed order Riemann-Liouville derivative written as
\begin{align}
\int_W \; \frac{^R\partial^{\alpha(y)}}{\partial t^{\alpha(y)}} u(t) \, p(dy) \, = \, \int_W \frac{1}{\Gamma (1-\alpha(y))} \frac{d}{dt} \int_0^t  u(s) \, (t-s)^{-\alpha(y)} \, ds \, p(dy),
\label{mazzao}
\end{align}
for which
\begin{align}
\mathcal{L} \left[ _0\mathcal{D}_t^{f, p} u(t) \right] (\lambda) \, = \, \mathds{E}  \lambda^{\alpha(Y)} \widetilde{u} (\lambda).
\end{align}
In order to obtain a more familiar expression of the distributed order fractional derivative one can set in \eqref{mazzao} $\alpha(y) = \alpha y$, for a constant $0<\alpha < 1$ and $p(dy)$ a probability measure on $(0,1)$.
\end{os}
\begin{os}
The operator
\begin{align}
_0\mathcal{D}^{f, p}_t u(t) - \mathds{E} \bar{\nu}(t, Y) u(0)
\label{to have caputo}
\end{align}
has the form of a regularized Riemann-Liouville derivative with a different kernel (as in \eqref{derivatakochu}) and Lemma \ref{dafare} shows that it may be viewed as a generalized distributed order Dzerbayshan-Caputo derivative.
In the logic of \eqref{to have caputo} we note that if $\mathds{E} b(Y)=0$, $\bar{\nu}(s, y) = s^{-\alpha(y)} / \Gamma(1-\alpha(y))$ we get that
\begin{align}
\mathcal{L} \left[ _0\mathcal{D}_t^{f, p} u(t) - \mathds{E}  \bar{\nu}(s, Y) \right] (\lambda) \, = \, \mathds{E}  \lambda^{\alpha (Y)} \widetilde{u}(\lambda) - \lambda^{-1} \mathds{E} \lambda^{\alpha(Y)} u(0).
\end{align}
\end{os}
The above arguments and Lemma \ref{dafare} inspire the following definition.
\begin{defin} 
\label{defreg}
Let $u$, $f$ and $\bar{\nu}$ be as in Definition \ref{dev1}. The regularized distributed order Riemann-Liouville derivative may be generalized as
\begin{align}
_c\mathsf{D}^{f, p}_t u(t) \, = \, & _c\mathcal{D}_t^{f, p} u(t) - u(c) \mathds{E} \bar{\nu}(t-c, Y) \notag \\
= \, &  \mathds{E}  b(Y) \frac{d}{dt} u(t) + \frac{d}{dt} \int_c^t u(s) \mathds{E}  \bar{\nu}(t-s, Y) ds - u(c) \mathds{E} \bar{\nu}(t-c, Y).
\label{regularized}
\end{align}
If $\mathds{E}b(Y)=0$ formula \eqref{regularized} make sense for a continuous function $u$, as it does for the classical distributed order fractional derivative (see \cite{meerdo} formula (2.6) and (2.7)).
\end{defin}

\subsection{The governing equations}
In what follows we work with subordinators with a density. In view of Theorem 27.7 in \citet{sato} we know that a sufficient condition for saying that a subordinator has a density is that $\nu(0, \infty) = \infty$ and that the function $s \to \bar{\nu} (s) = a+\nu(s, \infty)$ is absolutely continuous on $(0, \infty)$.
In such a framework we have the following.
\begin{te}
\label{teoremagoveq}
Let $\sigma^{f, y}(t)$ be a subordinator with Laplace exponent $f(\lambda, y)$. Suppose that $\nu((0, \infty), y) = \infty$, for all $y \in E$, and that the function $s \to \bar{\nu}(s,y) = a(y) + \nu((s, \infty), y)$ is absolutely continuous on $(0, \infty)$ for all $y \in E$. Let $\sigma^{f, p}$ be the corresponding distributed order subordinator with Laplace exponent $\mathds{E} f(\lambda, Y)$. Let $L^{f, p}(t)$, $t>0$, be the inverse of $\sigma^{f, p}$ with $l_t(B) = \Pr \ll L^f(t) \in B \rr$, $B$ Borel. We have that
\begin{enumerate}
\item  The subordinator $\sigma^{f, p}(t)$ has a density
\label{dens}
\item \label{renform} Denote the density of $\sigma^{f, p}(t)$ of point \eqref{dens} by $\mu_p(x, t)$. We have that $l_t(dx)$ admits the density 
\begin{align}
l_p(x, t) = \mathds{E} b(Y)\mu_p(t, x)+ \int_0^t \mu_p (s, x) \mathds{E} \bar{\nu}(t-s, Y)ds.
\end{align}
\item If $\mathds{E}b(Y)>0$ assume that $x \to \mu_p(x, t)$ is differentiable. We have that $\mu_p(x, t)$ solves
\begin{align}
\frac{\partial}{\partial t} q(x, t) \, = \, - \, _{\mathds{E}b(Y)t}\mathcal{D}_x^{f, p} q(x, t), \qquad x > t\, \mathds{E}b(Y), 0<t<\infty, 
\label{problb0}
\end{align}
subject to $q(x, 0)dx \, = \, \delta_0(dx)$ and $q(t\mathds{E}b(Y), t) =0$.

\item The density $l_p(x, t)$ of $L^{f, p}(t)$ solves
\begin{align}
_0\mathcal{D}^{f,p}_t q(x, t) \, = \, -\frac{\partial}{\partial x} q(x, t), \qquad\, 0 < t < \infty, \, \begin{cases}0<x<\frac{t}{\mathds{E}b(Y)}, \quad & \mathds{E}b(Y)>0,\\ 0<x<\infty, & \mathds{E}b(Y) = 0,  \end{cases}, 
\label{problinv}
\end{align}
subject to $q(0, t) \, = \, \mathds{E} \bar{\nu}(t, Y)$ and $q(x, 0)dx = \delta_0(dx)$. 
\end{enumerate}
\end{te}
\begin{proof}
\begin{enumerate}
\item From \eqref{47} and since $s \to \bar{\nu}(s, y)$ is absolutely continuous we have that
\begin{align}
& \int_W \bar{\nu} (s, y) p(dy) \, = \, \int_W a(y) p(dy)+ \int_W \int_s^\infty v(w, y)dw \, p(dy) \notag \\
 = \, & \int_W a(y)p(dy) + \int_s^\infty \int_W v(w, y) p(dy) \, dw,
\end{align}
where we denote by $v$ the density of $\nu$ with respect to the Lebesgue measure. Since $v$ is a density we have that $\mathds{E} v(s, Y)$ is a density and therefore $\int_s^\infty \mathds{E}  v(s, Y) ds$ is absolutely continuous. Furthermore if $\nu((0, \infty), y) = \infty$, for all $y \in E$ we clearly have that $\mathds{E} \nu((0, \infty), Y) = \infty$. By applying Theorem 27.7 of \cite{sato} we have proved this point.
\item First note that
\begin{align}
\mathcal{L} \left[ l_p (x, \puntomio) \right] (\lambda) \, = \, & -\frac{\partial}{\partial x} \int_0^\infty e^{-\lambda t} \Pr \ll \sigma^{f, p} (x) \leq t \rr \, dt \notag \\
= \, & \l \lambda^{-1} \int_W f(\lambda, y) p(dy) \r  \exp \ll -x \int_W f(\lambda, y) p(dy) \rr.
\label{laplacetempoinverso}
\end{align}
Now by \eqref{berncode} we write
\begin{align}
& \mathcal{L} \left[\mathds{E}  b(Y)\mu_p(\puntomio, x)+  \int_0^{\puntomio} \mu_p(s, x) \mathds{E} \bar{\nu}(\puntomio -s,Y)ds  \right] (\lambda) \notag \\
 = \, & \mathds{E}  b(Y) e^{-x\mathds{E} f(\lambda, Y)}+  \mathcal{L} \left[ \mu_p( \puntomio, x) * \mathds{E} \bar{\nu}(\puntomio, Y) \right] (\lambda) \notag \\
 = \, & \mathds{E}  b(Y) e^{-x\mathds{E} f(\lambda, Y)} +  e^{-x\mathds{E} f(\lambda, Y)} \l \frac{\mathds{E} f(\lambda, Y)}{\lambda} -\mathds{E}  b(Y) \r  \notag \\
= \, & \l \lambda^{-1} \int_W f(\lambda, y) p(dy) \r  \exp \ll -x \int_W f(\lambda, y) p(dy) \rr
\label{514}
\end{align}
and since \eqref{514} coincides with \eqref{laplacetempoinverso} we have proved this point.
\item Note that we need the differentiability of $x \to \mu_p(x, t)$ only if $\mathds{E} b(Y)>0$, indeed for $\mathds{E} b(Y)=0$ the operator \eqref{defriemvar} does not require the existence of the first derivative and therefore in this case it is well defined. In view of Lemma \ref{dafare} it is easy to show that the $x$-Laplace transform of the analytical solution to \eqref{problb0} solves the problem
\begin{align}
\frac{\partial}{\partial t} \widetilde{q}(\phi, t) \, = \, -\mathds{E}f(\phi, Y) \widetilde{q}(\phi, t)+\mathds{E} b(Y) q(t\mathds{E}b(Y), t).
\label{426}
\end{align}
By considering the boundary condition the last term of \eqref{426} disappears and by taking the Laplace transform with respect to $t$ we get that
\begin{align}
\lambda \widetilde{\widetilde{q}} (\phi, \lambda) -1 \, = \, -\mathds{E}f(\phi, Y) \widetilde{\widetilde{q}}(\phi, \lambda)
\label{426b}
\end{align}
where we have taken into account the initial condition. Formula \eqref{426b} implies
\begin{align}
\widetilde{\widetilde{q}}(\phi, \lambda) \, = \, \mathcal{L} \left[ \mathcal{L} \left[ q(x, t) \right] (\phi) \right] (\lambda) \, = \, \frac{1}{\lambda + \mathds{E} f(\phi, Y)}.
\end{align}
By observing that
\begin{align}
\int_0^\infty e^{-\lambda t} \mathds{E} e^{-\phi \sigma^{f, p}(t)} dt\, = \, \int_0^\infty e^{-t \l \lambda + \mathds{E} f(\phi, Y) \r} dt \, = \,\frac{1}{\lambda + \mathds{E} f(\phi, Y)},
\end{align}
the proof is complete.
\item  In view of point \eqref{renform} we know that
\begin{align}
l_p(x, t) \, = \, \mathds{E} b(Y)\mu_p(t, x) +\int_0^t \mu_p (s, x) \mathds{E} \bar{\nu}(t-s, Y) ds
\end{align}
and therefore the map
\begin{align}
t \to l_p(x, t) 
\end{align}
is differentiable. Note that if $\mathds{E} b(Y)=0$ we don't need to use the differentiability of $t \to \mu_p(t, x)$ since it is a density and the operator $_0\mathcal{D}^{f, p}$ exists for a continuous function, as well as the classical distributed order derivative (\cite{meerdo} page 217).

Now consider the Laplace-Laplace transform of the solution to \eqref{problinv}. In view of Lemma \ref{dafare} and formula \eqref{berncode}, we have
\begin{align}
\begin{cases}
\mathds{E} f(\lambda, Y) \widetilde{q}(x, \lambda  ) + \mathds{E} b(Y) q(x,0) \, = \, -\frac{\partial}{\partial x} \widetilde{q}(x, \lambda), \\
\widetilde{q}(0, \lambda) \, = \, \mathds{E} \frac{f(\lambda, Y)}{\lambda}-\mathds{E} b(Y)
\end{cases}
\end{align}
and thus by taking into account the conditions we get
\begin{align}
\widetilde{\widetilde{q}} (\phi, \lambda) \, = \, \frac{1}{\lambda} \frac{\mathds{E} f(\lambda, Y)}{\phi + \mathds{E} f(\lambda, Y)}.
\end{align}
Now by considering \eqref{514} we have that
\begin{align}
\mathcal{L} \left[ \mathcal{L} \left[ l_p(x, t) \right] (\lambda) \right] (\phi) \, = \, \frac{1}{\lambda} \frac{\mathds{E} f(\lambda, Y)}{\phi + \mathds{E} f(\lambda, Y)},
\end{align}
and this completes the proof.
\end{enumerate}
\end{proof}

\section{An application to slow diffusions}
An important application of distributed order fractional calculus is to model ultraslow diffusions, i.e. diffusions with mean square displacement $\overline{\l\Delta x\r^2} \sim C \log t$, $C>0$ (see, for example, \citet{chechprimo, chechsecondo, kochudo, meerstmod}). Roughly a diffusion is said to be slow if the mean square displacement behaves like $ct^\alpha$, for some $c>0$ and $\alpha < 1$. For subdiffusion see also \citet{magda2}. In \citet{kochudo} the reader can find a rigorous mathematical treatment of the equation 
\begin{align}
\int_0^1 \frac{\partial^\beta}{\partial t^\beta} q(x, t) p(d\beta) \, = \, \Delta q(x, t).
\label{dodiff}
\end{align}
Among other things, the author studied the behaviour of
\begin{align}
\overline{\l \Delta x \r^2} \, = \, \int_{\mathbb{R}^n} |x-u|^2 Z^\beta(x-u, t) du,
\end{align}
where $Z^\beta (x, t)$ is the fundamental solution to \eqref{dodiff}.
In this section we consider the fundamental solution to the equation
\begin{align}
_0\mathsf{D}_t^{f, p} q(x, t) \, = \, \Delta q(x, t), \qquad x \in \mathbb{R}^n, t>0,
\label{diffmia}
\end{align}
where $_0\mathsf{D}_t^{f, p}$ is defined in \eqref{regularized} and we study the mean square displacement. Our approach based on L\'evy mixing permits to carry out the study of \eqref{diffmia} by means of the so-called delayed Brownian motion (consult \citet{magda} for recent developments on sample paths properties of delayed Brownian motion).
We work as in the previous section with L\'evy measures such that $\nu((0, \infty), y) = \infty$ and $s \to \bar{\nu}(s, y)$ is absolutely continuous, for all $y \in E$. In the following theorem we study the behaviour of the mean square displacement related to the fundamental solution to \eqref{diffmia} under the additional assumption that $\mathds{E} f(\lambda, Y)$ is regularly varying at $0+$. A function $f$ is said to be regularly varying at $0+$ if $\lim_{\lambda \to 0} f(\lambda x)/f(\lambda)$ exists finite. In our case $\mathds{E} f(\lambda, Y) \in \mathrm{BF}$ and due to the L\'evy-Kintchine representation the limit is necessarily equal to $x^\alpha$ for $\alpha \in [0,1]$. The reader can consult \citet{bing} for a self-contained text on regular variation.
\begin{te}
Let $L^{f,p}(t)$ be as in Theorem \ref{teoremagoveq}.
Let
\begin{align}
\mathscr{M}(t) \, = \, \int_{\mathbb{R}^n} |x-u|^2 q(x-u, t) du
\end{align}
where $q(x, t)$ is the fundamental solution to \eqref{diffmia}. Let $\alpha$ be such that 
\begin{equation}
\lim_{\lambda \to 0}\mathds{E} f(\lambda x, Y)/\mathds{E} f(\lambda, Y) = x^\alpha
\end{equation}
and $B_n$ be the $n$-dimensional Brownian motion. We have the following results.
\begin{enumerate}
\item \label{1diff} $q(x, t)dx = \Pr \ll B_n \l L^{f, p} (t) \r \in dx \rr$, i.e. $q(x, t)$ can be viewed as a density of the one-dimensional marginal of the delayed Brownian motion (a Brownian motion $B_n$ time-changed with an independent $L^{f,p}(t)$)
\item $ \frac{1}{2n} \Gamma (1+\alpha) \mathscr{M}(t) \sim \frac{1}{\mathds{E} f(1/t, Y)}$ as $t \to \infty$ \label{2diff}
\item If $\mathds{E} a(Y) > 0$ then $\lim_{t \to \infty}\mathscr{M}(t) < \infty$ \label{3diff}
\item In general we have that 
\begin{equation}
\lim_{t \to \infty} \frac{t}{\frac{1}{2n} \Gamma(1+\alpha) \mathscr{M}(t)} \, = \, \mathds{E} b(Y) + \int_0^\infty \int_W  \bar{\nu}(s, y) p(dy) ds \, > 0,
\end{equation}
\label{4diff}
which is infinite or finite (but greater than zero), depending on the kernel $\mathds{E}\bar{\nu}(s, Y)$.
\end{enumerate}
\end{te}
\begin{proof}
First we prove that the fundamental solution to \eqref{diffmia}
admits the representation \eqref{1diff} and thus coincides with the law of the $n$-dimensional time-changed Brownian motion $B_n \l L^{f, p}(t) \r$, $t>0$.
In view of independence the law of $B_n \l L^{f, p}(t) \r$ is
\begin{align}
\Pr \ll B_n \l L^{f, p}(t) \r \in dx \rr  \, = \,& \int_0^\infty \Pr \ll B(s) \in dx  \rr  \, \Pr \ll L^{f, p}(t) \in ds \rr \notag \\
= \, & dx\int_0^\infty \frac{e^{-\frac{|x|^2}{4s}}}{\l 4\pi s\r^{n/2}} l_p(s, t) ds.
\label{casa}
\end{align}
In view of point (1) of Theorem \ref{teoremagoveq} we have that \eqref{casa} becomes
\begin{align}
& \Pr \ll B_n \l L^{f, p}(t) \r \in dx \rr /dx \, \notag\\ = \, & \int_0^\infty \frac{e^{-\frac{|x|^2}{4s}}}{(4\pi s)^{n/2}} \l b\mu_p(t, s)+ \int_0^t \mu_p(w, s) \mathds{E} \bar{\nu}(t-w, Y) dw \r ds
\end{align}
where $\mu_p (w, s)$ is the density of the subordinator $\sigma^{f, p}(s)$ such that
\begin{equation}
 L^{f, p} (t)  \, = \, \inf \ll s \geq 0: \sigma^{f, p}(s) > t \rr.
\end{equation}
By using Lemma \ref{dafare} we can consider the Fourier-Laplace transform of the analytical solution $q$ to \eqref{diffmia} (with $q(x, 0)dx = \delta_0 (dx)$) which reads
\begin{align}
\widehat{\widetilde{q}} (\xi, \lambda) \, = \, \mathds{E} \frac{f(\lambda, Y)}{\lambda} \frac{1}{\mathds{E} f(\lambda, Y) + | \xi |^2}
\end{align}
and by computing the Fourier-Laplace transform of \eqref{casa} we get that
\begin{align}
\mathcal{L} \left[ \mathds{E} e^{i\xi B_n \l L^{f, p}(\puntomio) \r} \right]  (\lambda) \, = \, & \mathcal{L} \left[ \int_0^\infty e^{-s | \xi |^2} l_p (s, \puntomio) ds \right] (\lambda) \notag \\
= \, & \lambda^{-1} \mathds{E} f(\lambda, Y) \frac{1}{\mathds{E} f(\lambda, Y) + | \xi |^2}.
\end{align}
We have proved \eqref{1diff}.
The mean square displacement in this case becomes
\begin{align}
\mathscr{M}(t) \,  = \, \overline{ \l \Delta x \r^2} \, = \,  & \int_{\mathbb{R}^n} |x-u|^2 q(x-u, t) du \notag \\
= \, &2n \int_0^\infty s \, l_p(s, t) ds \notag  \\
= \, &2n \, U^{f,p}(t)
\label{msdu}
\end{align}
where
\begin{align}
U^{f, p}(t) \, = \, \mathds{E} L^{f, p}(t) \, = \, \mathds{E}   \int_0^\infty \mathds{1}_{\ll \sigma^{f, p}(x) < t \rr} dx
\label{611}
\end{align}
is known in literature as the renewal function. The behavior of such a function has been studied under the assumption of regular variation. We recall that  if $f(x)$ is a Bernstein function regularly varying at $0+$ then for $\alpha \in [0,1]$
\begin{align}
\Gamma(1+\alpha) U^f(cx) \sim c^\alpha / f(1/x) \textrm{ as } x \to \infty.
\label{bertoinregvar}
\end{align}
Such a result may be found in \citet{bertoins}, Proposition 1.5. By applying \eqref{bertoinregvar} we write for \eqref{msdu}
\begin{align}
\frac{1}{2n}\Gamma(1+\alpha) \mathscr{M}(t) \, = \,\frac{1}{2n} \Gamma(1+\alpha) U^{f, p}(t)  \sim \frac{1}{\mathds{E} f(1/t, Y)} \textrm{ as } t \to \infty,
\label{attesa}
\end{align}
provided that the Bernstein function $\mathds{E} f(t, Y)$ is regularly varying at $0+$. We have proved \eqref{2diff}. Since for $\mathds{E}a(Y)>0$ we have that $\lim_{t \to \infty} 1/\mathds{E}f(1/t, Y)< \infty$ we have also proved \eqref{3diff}.

We observe that in general from result \eqref{attesa} we can write
\begin{align}
\lim_{t \to \infty} \frac{t}{\frac{1}{2n}\Gamma(1+\alpha)\mathscr{M}(t)} \, = \, & \mathds{E} b(Y)+ \lim_{t \to 0} \int_0^\infty e^{-st} \int_W  \bar{\nu}(s, y) \, p(dy) \, ds 
\label{rappt}
\end{align}
which clearly can not be zero but can be either finite or infinite, depending on the L\'evy measure and on $\mathds{E} a(Y)$. If $\mathds{E} a(Y)>0$ the limit \eqref{rappt} is clearly infinite. From \eqref{rappt} we can write
\begin{align}
\lim_{t \to \infty}\frac{t}{\frac{1}{2n}\Gamma(1+\alpha)\mathscr{M}(t)}\, = \, & \mathds{E} b(Y) + \int_0^\infty \int_W  \bar{\nu}(s, y) p(dy) ds,
\label{rapptfin}
\end{align}
which proves \eqref{4diff} and completes the proof of the Theorem.
\end{proof}
\begin{os}
Actually the limit \eqref{rapptfin} is infinite in most common cases and since it can not be zero we can state that the corresponding diffusions are subdiffusive or at most it can happen that $\frac{1}{2n}\Gamma(1+\alpha)\mathscr{M}(t) \sim Ct$, $C > 0$.
\end{os}
\begin{os}
As pointed out in Remark \ref{remfrac}, a fractional case can be for example $f(t, y) = t^{\beta(y)}$, $0 < \beta (y) < 1$. Suppose that $\beta (y) = \beta y$ with $p(dy)$ uniform in $(0, 1/C)$, $C>1$ and a constant $0<\beta<1 $. According to \eqref{attesa} we have for the mean square displacement
\begin{align}
\frac{1}{2n} \Gamma \l 1+\frac{\beta}{C} \r \mathscr{M}(t) \, \sim \, \frac{1}{\int_0^{1/C} \frac{C}{t^{\beta y }} dy} \, = \, \frac{\frac{\beta }{C} \log t}{1-t^{-\beta/C}}
\end{align}
and this is in accordance with Theorem 4.3 of \citet{kochudo}.
Suppose instead, for example, that
\begin{align}
f(t, y) \, = \, \log(1+t/y) \, = \, \int_0^\infty \l 1-e^{-st} \r \frac{e^{-ys}}{s}ds, \qquad y>0,
\end{align}
which is the Laplace exponent of a gamma subordinator, parametrized by $y$.
If $p(dy) = \delta_y$ one has from \eqref{rapptfin} that
\begin{align}
\lim_{t \to \infty} \frac{t}{\frac{1}{2n} \Gamma(2) \mathscr{M}(t)} \, = \, \int_0^\infty e^{-ys}ds \, = \, \frac{1}{y}.
\end{align}
Suppose instead that $p(dy) = \beta \gamma^\beta y^{-\beta -1} dy$ for $y \geq \gamma >0$, $\beta > 0$, formula \eqref{rapptfin} yields
\begin{align}
\lim_{t \to \infty} \frac{t}{\frac{1}{2n} \Gamma(1+\alpha) \mathscr{M}(t)} \, = \, \frac{1}{\gamma} \frac{\beta}{\beta +1}.
\end{align}
\end{os}

\section*{Acknowledgements}
Thanks are due to the Referees and to the Associate Editor whose remarks and suggestions have considerably improved a previous draft of the paper.

\end{document}